\documentclass[letterpaper, 10 pt, conference]{ieeeconf}  

\IEEEoverridecommandlockouts                              

\usepackage{graphicx} 
\usepackage{epstopdf}

\usepackage{subfigure}
\usepackage{amsmath} 
\usepackage{amssymb}  
\usepackage{algorithm}
\usepackage{algpseudocode}

\usepackage[noadjust]{cite}

\newtheorem{lemma}{Lemma}
\newtheorem{theorem}{Theorem}

\title{\LARGE \bf
Surrogate Optimal Control for Strategic Multi-Agent Systems
}

\author{Pedro Hespanhol and Anil Aswani
\thanks{$^{1}$P. Hespanhol and A. Aswani are with the Department of Industrial Engineering and Operations Research, University of California, Berkeley, CA 94720 USA 
        {\tt\small pedrohespanhol@berkeley.edu, aaswani@berkeley.edu}}%
}

\begin{document}

\maketitle
\thispagestyle{empty}
\pagestyle{empty}

\begin{abstract}
This paper studies how to design a platform to optimally control constrained multi-agent systems with a single coordinator and multiple strategic agents. In our setting, the agents cannot apply control inputs and only the coordinator applies control inputs; however, the coordinator does not know the objective functions of the agents, and so must choose control actions based on information provided by the agents. One major challenge is that if the platform is not correctly designed then the agents may provide false information to the coordinator in order to achieve improved outcomes for themselves at the expense of the overall system efficiency. Here, we design an interaction mechanism between the agents and the coordinator such that the mechanism: ensures agents truthfully report their information, has low communication requirements, and leads to a control action that achieves efficiency by achieving a Nash equilibrium. In particular, we design a mechanism in which each agent does not need to posses full knowledge of the system dynamics nor the objective functions of other agents.  We illustrate our proposed mechanism in a model predictive control (MPC) application involving heating, ventilation, air-conditioning (HVAC) control by a building manager of an apartment building. Our results showcase how such a mechanism can be potentially used in the context of distributed MPC.
\end{abstract}


\section{Introduction}

Many systems have dynamics influenced by agents, including power systems \cite{venkat2008distributed}, communication networks \cite{kelly1998rate}, water systems \cite{negenborn2009distributed}, and heating, ventilation, and air-conditioning (HVAC) automation \cite{aswani2012energy}. These systems are characterized by information flows and the order of computations. For cooperative agents, various distributed model predictive control (MPC) schemes have been designed. A system-level control policy was obtained by aggregating locally-computed inputs \cite{farina2011distributed,alessio2011decentralized}, and central platforms that compute a control based upon information sent by agents have also been designed \cite{ferrari2009model}. 

Distributed control with strategic agents is less well-studied. The competitive nature of agents and asymmetries of information reward tactical behavior, ultimately leading to instability or poor performance \cite{venkat2005stability,rawlings2009,mintz2018control,johari2004efficiency}. We focus our attention on the case where equilibrium behavior can be described as a Nash equilibrium of some non-cooperative game \cite{neck1987conflict} that may be inefficient \cite{cohen1998cooperation}. A common way to overcome such inefficiencies is to force agents to coordinate their goals with the system-wide goal \cite{marden2009overcoming}, \cite{li2010designing}. However, this approach requires strong assumptions that the agents' utility functions are common knowledge and/or agents are honest when transmitting information \cite{stewart2010cooperative}. Another line of work \cite{ratliff2012pricing,coogan2013energy} provides pricing schemes to induce or manage agents' behavior in the equilibrium.
\subsection{Contributions}	

In this paper, we study the case of strategic agents under weaker assumptions than past work like \cite{shamma2008cooperative,stewart2010cooperative}. In particular, the agents exchange information only with a central platform that is responsible for the control decision. Our goal is to design the interaction mechanism to ensure not only efficiency of the resulting control policy but also honest reporting from the agents. Originally, the study of such mechanisms \cite{kelly1997charging} was concerned with the design of incentives to ensure efficient allocation of commodities amongst market participants, whilst ensuring truthfulness. The classical VCG mechanisms \cite{vickrey1961counterspeculation,clarke1971multipart,groves1973incentives} are an example of such. Our first contribution lies in providing a mechanism that enjoy those properties when applied to an optimal control setting.

A major hurdle in implementing such mechanisms is their steep communication needs \cite{yang2007vcg,johari2004efficiency,farhadi2018surrogate}. But, minimal strategy spaces that elicit efficient Nash equilibrium in convex environments have been developed \cite{reichelstein1988game}.  A second contribution of our work is to provide communication protocols that are of low complexity order: We avoid communicating the entire utility function by the agents and instead resort to vector-valued messages inspired by surrogate optimization \cite{yang2007vcg}, \cite{johari2009efficiency}. Hence, our goal in this paper is to provide a platform where (i) agents provide low-dimensional information, (ii) agents are honest, and (iii) an efficient control policy is implemented in the Nash equilibrium.

Lastly, we demonstrate the practical usefulness of our designed platform by conducting a simulation analysis of HVAC automation \cite{ma2011distributed}. The situation we consider involves an apartment building where each apartment has its own preferences on desired room temperature versus the amount of energy consumption. The thermal dynamics of each apartment are coupled, and more efficient control is possible through coordination. Our simulations quantify the performance improvement possible through the use of our central platform in coordinating agents. In fact, this HVAC setup is similar to the setup in \cite{coogan2013energy}. However, a major difference is that in \cite{coogan2013energy} the central platform knows each agents' utility function and can set prices on the control inputs to induce agents' behavior. In contrast, we allow the agents to be strategic with respect to how they communicate information about their utility function to the central platform.

\subsection{Outline}

Sect. \ref{sec2} defines the system model, and Sect. \ref{sec3} defines the mechanism and how agents interact with it. In Sect. \ref{sec4}, we provide a Nash equilibrium characterization of the agents' equilibrium behavior. We conclude with Sect. \ref{sec5}, where we provide a case study in the context of HVAC automation.
\vfill

\section{System Model}

\label{sec2}

Consider a system which obeys linear dynamics
\begin{equation}
x_{k+1} = Ax_{k} + Bu_{k}
\end{equation}
where $x_k \in \mathbb{R}^{n}$ is the state vector, and $u_k \in \mathbb{R}^{m}$ is the input signal. Suppose this system is composed of $I$ interconnected and non-overlapping subsystems that are each associated to an agent. Let $[I]$ denote the set $\{0,...,I\}$. We let $x^{(i)}_{k} \in \mathbb{R}^{n_{i}}$ denote the state vector of subsystem $i$ at period $k$. Then we have $x_{k} = (x^{(1)}_{k},...,x^{(I)}_{k} )$ and $\sum_{i= 1}^{I}n_i=n$. In addition, we can also partition the inputs where $u^{(i)}_{k} \in \mathbb{R}^{m_i}$. Note it follows that $u_k = (u^{(1)}_{k},...,u^{(I)}_{k})$ and $\sum_{i= 1}^{I}m_i=m$.

\subsection{Agent Model}

The diagonal block $A_{ii}$ of $A$ gives the subsystem dynamics for the $i$-th agent. Influence by other agents is described by off-diagonal blocks $A_{ij}$ of $A$ when subsystem $j$ impacts $i$. We assume agent $i$'s input only affects states in their subsystem; hence, the input matrix $B = \mathrm{diag}(B_{1},...,B_{I})$ is block-diagonal. Let $\mathcal{N}_{i}$ be the set of neighboring subsystems of subsystem $i$. Then the dynamics for the $i$-th subsystem is
\begin{equation}\label{agent_subsys}
\textstyle x^{(i)}_{k+1} = A_{ii}x^{(i)}_{k} + B_{i}u^{(i)}_{k} + \sum_{j \in\mathcal{N}_{i}} A_{ij}x^{(j)}_{k}.
\end{equation}
We assume each agent only knows their own local dynamics $A_{ii}$ and $B_i$. Since agents do not know the $\sum_{j \in\mathcal{N}_{i}} A_{ij}x^{(j)}_{k}$ part of their dynamics, a central platform is needed through which each agent can receive this information. 

Each subsystem has state $\mathbb{X}_{i} = \{ G^{(i)}_{x}x^{(i)} \leq g^{(i)}_{x}  \}$ and input constraints $\mathbb{U}_{i} = \{ G^{(i)}_{u}u{(i)} \leq g^{(i)}_{u}  \}$ that are polytopes containing the origin. Here, $\{G^{(i)}_{x},G^{(i)}_{u}\}_{i=1}^{I}, \{g^{(i)}_{x},g^{(i)}_{u}\}_{i=1}^{I}$ are matrices and vectors with appropriate dimensions, respectively. Lastly, each agent $i$ has their own cost function
\begin{equation}
\textstyle\textbf{V}_{i}(x^{(i)},u^{(i)})) = g_i(x^{(i)}_T) + \sum_{k=0}^{T-1}l_{i}(x^{(i)}_{k},u^{(i)}_{k})  
\end{equation}
where $l_{i}(\cdot,\cdot)$ and $g_i(\cdot)$ are stage and terminal costs, and $T$ is control horizon. We assume each agent's stage and terminal costs are strictly convex, differentiable, and take their minimum at the origin. The cost function is the agents' private information, and their goal is to minimize it.


\subsection{Principal Model}

The central platform is operated by a coordinator that we call the \emph{principal}. We assume the principal has complete knowledge about the dynamics of the system (i.e., matrices $A$ and $B$) and constraints (i.e., the sets $\mathbb{X}_{i},\mathbb{U}_{i} , \forall i \in [I]$), and importantly the principal is who gets to apply a control input to the entire system (restated, the agents do not directly provide control inputs). In this framework, if the principal knew the objective function of each agent, then they could compute a control sequence by solving the following convex optimal control problem (OCP-T):
\begin{align}
& \min_{x,u}\ \textstyle\sum_{i =1}^{M}(g_i(x^{(i)}_T) + \sum_{k =0}^{T-1}l_{i}(x^{(i)}_{k},u^{(i)}_{k})) \label{TRUE-OCP} \nonumber\\
&\ \begin{aligned}
\mathrm{s.t.}\ & x_{k+1} = Ax_{k} + Bu_k,&& \forall k \in [T-1]  \\
& x^{(i)}_{k} \in \mathbb{X}_{i},  u^{(i)}_{k} \in \mathbb{U}_{i}, &&\forall i \in [I], k \in [T]  \\
& x^{(i)}_{0} = \bar{x}^{(i)}_{0}, &&\forall i \in [I]
\end{aligned}
\end{align}

Throughout the paper we let $(x^{*},u^{*})$ denote the optimal solution of (OCP-T), which we call the \textit{efficient trajectory}. However, solving this problem is not possible for the principal, since it does not know the objective functions of each agent. It then needs to elicit information from each agent. The need of information gives birth to two major issues, which are the central focus of this work: (1) The agents may not be able/not desire to transmit their entire cost functions to the principal, as each cost function is infinite-dimensional and their private information; (2) The agents are strategic and may be not tell the truth. Therefore in order for the principal to solve (OCP-T) it also needs to design a mechanism that provides incentives to each agent to tell the truth. Hence, the principal is faced with both an optimal control problem and a mechanism design problem.

\section{Mechanism Specification}
\label{sec3}

As described in the previous section, the principal's goal is to solve (OCP-T). Towards that goal, the principal resorts to approximate the objective function based on a finite number of parameters that the agents can report, and then the principal will minimize this approximated function. 

\subsection{Definition of a Mechanism}

Let a \emph{mechanism} $\mathcal{M}$ be a tuple $(M_1,...,M_{I},z,p)$, where $M_i$ is the set of allowable messages agent $i$ can send to the principal, and $z(\cdot)$ is the outcome function that determines the outcome $z(m)$ for any message profile $m = (m_1,...,m_I) \in M_1 \times ... \times M_I$. Here, the outcome function maps a message profile $m$ to a state/input trajectory $(x,u)$:
\begin{equation}
z(m): ( M_1 \times ... \times M_I) \rightarrow \mathbb{R}^{n\times (T+1)} \times \mathbb{R}^{m\times T}
\end{equation}
where $z_{i}(m)$ refers to the state/input trajectory associated with agent's $i$ subsystem. Next, we define $p(m)$ to be a non-negative vector of ``fees'' for each agent. 

The mechanism $\mathcal{M}$ together with the cost functions of each agents $(\textbf{V}_{i})_{i=1}^{I}$ induce a game $\textbf{N} = (\mathcal{M},(\textbf{V}_{i})_{i=1}^{I})$ among the agents. We define the Nash equilibrium (NE) of this game as a message profile $m^{*}$ such that
\begin{multline}
\textbf{V}_{i}(z_{i}(m^{*}_i,m^{*}_{-i})) + p_i(m^{*}_i,m^{*}_{-i}) \leq \\ \textbf{V}_{i}(z_{i}(m_i,m^{*}_{-i})) + p_i(m_i,m^{*}_{-i}),
\end{multline}
for all $m_i \in M_i$ and $i\in[I]$, where the compact notation $m^{*}_{-i}$ denotes the vector of messages from all agents except $i$. The fee $p_i$ increases costs for agent $i$, which is undesirable since agents are minimizing. The goal of the principal is to design the mechanism such that the efficient trajectory $(x^{*},u^{*})$ can be implemented as the Nash equilibrium of the game \textbf{N}. Implementation means that the trajectory corresponding to the Nash equilibrium of the game induced by the mechanism is equal to the efficient trajectory.

\subsection{Low-Communication Mechanism}

We specify our low-communication mechanism as follows: Each agent $i$ reports messages $m_i \in M_i$ of the form
\begin{equation} \label{message}
m_i = (v^{(i)},w^{(i)},\tilde{\lambda}^{i}, \tilde{J}^{(i)},\tilde{x}^{(i)})
\end{equation}
where $v^{(i)} \in \mathbb{R}^{n_i \times (T+1)}$ are weights for every state of subsystem $i$ for each stage; $w^{(i)} \in \mathbb{R}^{m_i \times T}$ are weights for every control input of subsystem $i$ for each stage; $\tilde{\lambda}^{i} \in \mathbb{R}^{n_i \times T}$ are weights representing the ``sensitivity'' of agent $i$ dynamics in cost function for each stage; $\tilde{J}^{(i)} = (\{\underline{x}^{(i)}_{k},\bar{x}^{(i)}_{k}\}_{k=0}^{T},\{\underline{u}^{(i)}_{k},\bar{u}^{(i)}_{k}\}_{k=0}^{T})$ is vector of bounds for states/inputs; and $\tilde{x}^{(i)}= (\tilde{x}^{(i)}_{0},...,\tilde{x}^{(i)}_{T})$ is a reference trajectory for the states of subsystem $i$. Restated, each agent provides some open-loop trajectory coupled with state and input bounds, as well as scalars measuring the ``impact'' of states, inputs, and dynamics in its cost function.

In addition, the principal announces a single real-valued function $f:\mathbb{R} \times \mathbb{R} \rightarrow \mathbb{R}$ to all agents to be used as a surrogate function for their cost functions. Namely, for each agent $i$ the principal forms the surrogate function
\begin{multline}
f_{i}(v^{(i)}_{k},w^{(i)}_{k},x^{(i)}_{k},u^{(i)}_{k}) = \\\textstyle\sum_{j=1}^{m_i}(f(x^{(i)}_{j,k};v^{(i)}_{j,k}) + f(u^{(i)}_{j,k};w^{(i)}_{j,k})) \text{ for } k\in[T-1]
\end{multline}
as an approximation of the agent's stage cost $l_i(\cdot,\cdot)$. The notation $f(\cdot;\cdot)$ indicates the second argument is a parameter of the function and not a variable. We further only consider functions $ f(\cdot; v)$ that are strictly convex for all possible parameters $v \in \mathbb{R}$. Lastly, for simplicity we let the principal announce the same function for both states and inputs. But one could consider different functions for states and inputs -- the key property being that it is the same function for all agents. Then the principal forms the surrogate function
\begin{equation}
F_i(v({i)}_T,x^{(i)}_T) = \textstyle\sum_{j=1}^{m_i} f(x^{(i)}_{j,T};v^{(i)}_{j,T})
\end{equation}
as an approximation of agent's terminal cost $g_i(\cdot)$. Based on a message profile $m$, the principal formulates the following surrogate optimal control problem (OCP-S):
\begin{align}
&\min_{x,u}\ \sum_{i =1}^{I}(F_i(v({i)}_T,x^{(i)}_T) + \sum_{k =0}^{T-1}f_{i}(v^{(i)}_{k},w^{(i)}_{k},x^{(i)}_{k},u^{(i)}_{k}))\nonumber
\\
&\ \begin{aligned}\mathrm{s.t.}\ &x_{k+1}=Ax_k+Bu_k,&&\forall k \in [T-1] \\
& x^{(i)}_{k} \in \mathbb{X}_{i},  u^{(i)}_{k} \in \mathbb{U}_{i}, &&\forall i \in [I], k \in [T]\\
&  (x^{(i)}_{k},u^{(i)}_{k}) \in \tilde{J}^{(i)}_k, &&\forall k \in [T] \\
& x^{(i)}_{0} = \bar{x}^{(i)}_{0}, &&\forall i \in [I]
\end{aligned} \label{MECH-OCP}
\end{align}
where we explicitly consider the desired operational bounds reported by each agent $\tilde{J}^{(i)}_k$ for every stage $k$. 

Since $f$ is strictly convex, this optimization problem has an unique solution that we call $(x(y^{*}),y^{*})$. Note that this notation means the $y^*$ are the optimal inputs for OCP-S. Then, given a message profile $m$, we have that $z(m) =(x(y^{*}),y^{*})$.  That is, the outcome function of the mechanism $z(m)$ outputs exactly the state/input trajectory of the optimal solution of OCP-S. We also define $\lambda^{*(i)}$ to be the optimal lagrange multipliers associated with Eq.~\ref{agent_subsys} for every agent $i$.

Now, suppose the game is repeatedly played with the same initial condition $\bar{x}_0$. At first, this mechanism is run for one round, meaning the principal collected some message $m$, and solved OCP-S once. Then, before the next round the principal sends the following reference trajectory $c^{(i)}$ to agent $i$:
\begin{equation}
\textstyle c^{(i)}_{k} = \sum_{j \in\mathcal{N}_{i}} A_{ij}\tilde{x}^{(j)}_{k}
\end{equation}
for $k \in[T-1]$, where $\tilde{x}^{(j)}$ is part of the message $m_{j}$ as per (\ref{message}). Observe that the reference trajectory sent to agent $i$ does not depend upon solving OCP-S. Moreover the principal will assign the following fees to each agent:
\begin{align} \label{transfer}
& p_i = \sum_{k=0}^{T-1}\Lambda_{-i,k}^{\top}(x^{(i)}_{k}(y^{*}) - \hat{x}^{(i)}_{k}(c^{(i)})) \text{ } + \nonumber\\
& ||\tilde{x}^{(i)} - x^{(i)}(y^{*})||^{2}_{2} + ||\tilde{\lambda}^{(i)} - \lambda^{*(i)}||^{2}_{2}
\end{align}
where $\hat{x}^{(i)}(c^{(i)})$ is a state reference trajectory computed by the principal for agent $i$ given that the other agents behave according to $c^{(i)}$. For example, the principal can solve another round of OCP-S but now excluding agent $i$'s contribution to the objective function in order to obtain $\hat{x}^{(i)}(c^{(i)})$ (in a way akin to VCG mechanisms~\cite{groves1973incentives}). The key observation here is that the reference trajectory $\hat{x}^{(i)}(c^{(i)})$ does not depend on the message sent by agent $i$. The first term of the fee penalizes deviations of the computed optimal state trajectory $x^{(i)}(y^{*})$ from $\hat{x}^{(i)}(c^{(i)})$.  The second term penalizes mismatches between the reported $\tilde{x}^{(i)}$ and the optimal state trajectory $x^{(i)}(y^{*})$. The third term penalizes deviations from the reported sensitivity vector $\tilde{\lambda}^{(i)}$ and the optimal lagrange multipliers  $\lambda^{*(i)}$ of OCP-S associated with the dynamics of agent $i$. Lastly the vectors $\Lambda_{-i,k}$ are computed by the principal as follows:
\begin{equation}
\Lambda_{-i,k}^{\top} = \sum_{j: i \in\mathcal{N}_{j}} \tilde{\lambda}^{(j)\top}_{k}A_{ji}  \text{, } \forall k\in[T-1]
\end{equation}

where we, once again, note that this vector does not depend on the message sent by agent $i$.

\section{equilibrium Characterization}

\label{sec4}

With the mechanism defined, we can now characterize the equilibrium behavior of agents interacting via this mechanism. The goal of this section is to characterize the Nash equilibrium (NE) of the reulting game, which is a message profile $m^{*}$. We start by first analyzing the properties of such equilibrium, and then we show it actually exists. Our analysis begins by showing that in a NE $m^{*}$, each agent $i$ reports a specific type of state reference trajectories to the principal.

\begin{lemma}
Let $m^{*}=(v^{*},w^{*},\tilde{\lambda}^{*},\tilde{J}^{*},\tilde{x}^{*})$ be a NE of the game induced by the mechanism. Then every agent $i \in I$ reports $\tilde{x}^{(i)*} = x^{(i)}(y^{*})$ and $\tilde{\lambda}^{(i)*} = \lambda^{(i)*}$. In addition, the principal sends the following references to the agents:
\begin{equation}
\textstyle c^{*(i)}_{k} = \sum_{j \in\mathcal{N}_{i}} A_{ij}x^{(j)}_{k}(y^{*}) \text{, } \forall k\in[T-1]
\end{equation}
\end{lemma} 

\begin{proof}
Suppose all agents adhere to the message profile $m^{*}$, except agent $i$ which reports some message $m_{i} = (v^{(i)},w^{*(i)},\tilde{\lambda}^{(i)},J^{*(i)},\tilde{x}^{(i)})$. Since $m^{*}$ is a NE, this deviation should give a higher cost for agent $i$, that is:
\begin{multline}
\textbf{V}_{i}(z_{i}(m_i,m^{*}_{-i})) + p_i(m_i,m^{*}_{-i}) \geq \\ \textbf{V}_{i}(z_{i}(m^{*}_i,m^{*}_{-i})) + p_i(m^{*}_i,m^{*}_{-i}).
\end{multline}
Now, observe that the outcome function $z(m)$ only depends on the $(v^{*},w^{*},\tilde{J}^{*})$ components of the message $m^{*}$. Then substituting into (\ref{transfer}) gives that
\begin{align}
& ||\tilde{\lambda}^{(i)} - \lambda^{*(i)}||^{2}_{2} + ||\tilde{x}^{(i)} - x^{(i)}(y^{*}))||^{2}_{2} \nonumber\\
& \geq ||\tilde{\lambda}^{*(i)} - \lambda^{*(i)}||^{2}_{2} + ||\tilde{x}^{*(i)} - x^{(i)}(y^{*}))||^{2}_{2}
\end{align}
for all possible sensitivities and state trajectory reports $(\tilde{\lambda}^{(i)},\tilde{x}^{(i)})$. Hence $(\tilde{\lambda}^{*(i)},\tilde{x}^{*(i)})$ is the solution of the following minimization problem:
\begin{equation}
\min_{(\tilde{\lambda}^{*(i)},\tilde{x}^{*(i)})}\{||\tilde{\lambda}^{(i)} - \lambda^{*(i)}||^{2}_{2} + ||\tilde{x}^{(i)} - x^{(i)}(y^{*})||^{2}_{2}\}
\end{equation}
which achieves the minimum when $(\tilde{\lambda}^{*(i)},\tilde{x}^{*(i)}) = (\lambda^{*(i)}x^{(i)}(y^{*}))$. Then by definition of $c^{(i)}$ it directly follows that
\begin{equation}
\textstyle c^{*(i)}_{k} = \sum_{j \in\mathcal{N}_{i}} A_{ij}x^{(j)}_{k}(y^{*})  \text{, } \forall k\in[T-1]
\end{equation}
\end{proof}

Next, observe that each agent can only ``measure'' the impact of other subsystems in its dynamics via the reference signal $c^{(i)}$ that is sent by the principal. We say an state/input sequence $(\check{x}^{(i)},\check{u}^{(i)})$ is feasible for agent $i$ if it is feasible for the agent's subsystem given the reference $c^{(i)}$. We proceed to show that given the reference $c^{(i)}$, any feasible state/input sequence $(\check{x}^{(i)},\check{u}^{(i)})$ can be achieved by agent $i$. That is, agent $i$ can send a message that makes the principal compute the input $y^{*(i)} = \check{u}^{(i)}$ and $x^{(i)}(y^{*}) = \check{x}^{(i)}$ as it solves the problem OCP-S, given that OCP-S is feasible. 
\begin{lemma}
For any agent $i$, given a feasible state/input sequence $(\check{x}^{(i)},\check{u}^{(i)})$ there exists a message $\bar{m}_{i}$ such that $z_{i}(\bar{m}_{i}, m_{-i}) = (\check{x}^{(i)},\check{u}^{(i)})$ for all possible messages of the other agents $m_{-i}$, given that the resulting optimization problem (OCP-S) is feasible for $(\bar{m}_{i},m_{-i})$.
\end{lemma}

\begin{proof}
Fix some agent $i$ and a feasible state/input sequence $(\check{x}^{(i)},\check{u}^{(i)})$. We prove this lemma by constructing the message $\bar{m}_i$. Specifically, suppose agent $i$ chooses $\underline{x}_k^{(i)} = \bar{x}_k^{(i)} = \check{x}_k^{(i)}$ and  $\underline{u}_k^{(i)} = \bar{u}_k^{(i)} = \check{u}_k^{(i)}$. This choice constrains OCP-S to require that $y^*(i) = \check{u}^{(i)}$ and $x^{(i)}(y^{*}) = \check{x}^{(i)}$. Then for any message $m_{-i}$, OCP-S is either infeasible or returns the desired solution for agent $i$, regardless of the message of other agents.
\end{proof}

What this lemma implies is that given the Nash equilibrium message profile $m^{*}$, agent $i$ can unilaterally deviate in such a way that the principal will compute $(\check{x}^{(i)},\check{u}^{(i)})$ as part of the optimal solution, as long as OCP-S is feasible. We proceed in writing the agent's optimal control problem (OCP-A) in equilibrium:
\begin{align} 
&\min_{x^{(i)},u^{(i)}} \textstyle g_i(x^{(i)}_T) + \sum_{k =0}^{T-1}l_{i}(x^{(i)}_{k},u^{(i)}_{k}) + p^{*}_i(x^{(i)})\nonumber\\
&\ \begin{aligned}
\mathrm{s.t.}\ & x^{(i)}_{k+1} = A_{ii}x^{(i)}_{k} + B_{i}u^{(i)}_{k} + c^{*(i)}_{k}, &&\forall k \in [T-1] \\
& x^{(i)}_{k} \in \mathbb{X}_{i},  u^{(i)}_{k} \in \mathbb{U}_{i}, &&\forall k \in [T]\\
& x^{(i)}_{0} = \bar{x}^{(i)}_{0}
\end{aligned}\label{AGENT-OCP}
\end{align}
where the equilibrium fee, according to Lemmas 1 and 2 is given by:
\begin{equation}
p^{*}_i(x^{(i)}) = \sum_{k=0}^{T-1}\Lambda_{-i,k}^{*\top}(x^{(i)}_{k} - \hat{x}^{(i)}_{k}(c^{(i)}))
\end{equation}
where $\Lambda_{-i,k}^{*\top} = \sum_{j: i \in\mathcal{N}_{j}} \lambda^{*(j)\top}_{k}A_{ji}  \text{, } \forall k\in[T-1]$.

Thus in order for a message profile $m^{*}$ to be a NE, we must have that the optimal solution $(x(y^{*}),y^{*})$ for OCP-S must also be an optimal solution for each agents' OCP-A. Since both OCP-S and OCP-A are convex problems, it is enough to require that $(x^{*(i)}(y^*),y^{*(i)})$ satisfy the KKT conditions for OCP-A for every agent $i$. Next we present our main theorem, which shows that the efficient trajectory $(x^*,u^*)$ can be implemented as a Nash equilibrium of the game induced by the mechanism:

\begin{theorem}
\textbf{(Implementability): } The unique efficient trajectory can be supported as a Nash equilibrium $m^{*}$ of the game induced by the mechanism, that is $(x(y^{*}),y^{*})= (x^{*},u^{*})$. In addition the equilibrium messages satisfy
\begin{equation} \label{N_mesg}
\begin{aligned}
&f'(x^{*(i)}_{j,k};v^{*(i)}_{j,k} )=  \frac{\partial l_{i}(x^{*(i)}_{k},u^{*(i)}_{k}) }{\partial x^{(i)}_{j,k}}, &\forall i,j,k\\
&f'(x^{*(i)}_{j,T};v^{*(i)}_{j,T})=  \frac{\partial g_{i}(x^{*(i)}_{T}) }{\partial x^{(i)}_{j,T}}, &\forall i,j\\
&f'(u^{*(i)}_{h,k};w^{*(i)}_{h,k}) =\frac{\partial l_{i}(x^{*(i)}_{k},u^{*(i)}_{k}) }{\partial u^{(i)}_{h,k}}, &\forall i,h,k\\
&\tilde{x}^{i} = x^{*(i)}\text{, } \tilde{\lambda}^{i} = \lambda^{*(i)} &\forall i\\
&\tilde{J}^{(i)} = (\{-\infty,+\infty\}_{k=0}^T, \{-\infty,+\infty\}_{k=0}^T), &\forall i
\end{aligned}
\end{equation}
where $f'(\cdot)$ denotes the derivative of $f(\cdot)$.
\end{theorem}

\begin{proof}
First, note OCP-T is an ``aggregation'' of each agent's problem: Instead of optimizing each agent separately with references $c^{(i)}$ for the neighbors, we optimize all agents at once. The KKT stationarity conditions  for multipliers $(\nu,\gamma_x,\gamma_u)$ of OCP-T, associated with the dynamics, state and input constraints respectively, are
\begin{equation}  \label{tem_eq1}
\begin{aligned}
&\textstyle\frac{\partial l_{i}(x^{*(i)}_{k},u^{*(i)}_{k}) }{\partial x^{(i)}_{j,k}} + \nu^{(i)\top}_{k}{A_{ii}}_{,j} - \nu^{(i)}_{k-1,j} + \\
&\qquad\qquad\qquad\textstyle\gamma^{(i)\top}_{x}G^{(i)}_{x,j,k}+\sum_{\ell\in\mathcal{N}_i}{\nu^{(\ell)\top}_{k}A_{\ell i,j}} = 0\\
& \textstyle\frac{\partial g_{i}(x^{*(i)}_{T}) }{\partial x^{(i)}_{j,k}} - \nu^{(i)}_{T-1,j} + \gamma^{(i)\top}_{x}G^{(i)}_{x,j,T} = 0\\
&\textstyle\frac{\partial l_{i}(x^{*(i)}_{k},u^{*(i)}_{k}) }{\partial u^{(i)}_{h,k}} + \nu^{(i)\top}_{k}{B_{ii}}_{,h} + \gamma^{(i)\top}_{u}G^{(i)}_{u,h,k} = 0
\end{aligned}
\end{equation}
for all $j\in\{1,\ldots,n_i\}$, $h\in\{1,\ldots,m_i\}$ and $k\in \{1,...,T-1\}$. On the above we use the notation $A_{li,j}$ to denote the column $j$ of matrix $A_{li}$. In addition we let $B_{ii,h}$ denote the column h of $B_{ii}$. Similarly, $G_{\cdot,j,k}$ represents the column $j$ of the stage k constraints matrix $G_{\cdot,k}$. Now, if the messages follow (\ref{N_mesg}), then it is easy to see that $(x^{*},u^{*})$ satisfy the KKT conditions of OCP-S:
\begin{equation} \label{tem_eq1}
\begin{aligned}
& f'(x^{*(i)}_{j,k};v^{*(i)}_{j,k}) + \lambda^{(i)\top}_{k}{A_{ii}}_{,j} - \lambda^{(i)\top}_{k-1,j} + \\
&\qquad\qquad\qquad\qquad \beta^{(i)\top}_{x}G^{(i)}_{x,j,k}+\sum_{\ell\in\mathcal{N}_i}{\lambda^{(\ell)\top}_{k}A_{\ell i,j}} = 0\\
&\textstyle f'(x^{*(i)}_{j,T};v^{*(i)}_{j,T}) - \lambda^{(i)}_{T-1,j} + \beta^{(i)\top}_{x}G^{(i)}_{x,j,T}  = 0\\
&f'(u^{*(i)}_{h,k};w^{*(i)}_{h,k}) + \lambda^{(i)\top}_{k}{B_{ii}}_{,h} + \beta^{(i)\top}_{u}G^{(i)}_{u,h,k} = 0
\end{aligned}
\end{equation}
for all $j\in\{1,\ldots,n_i\}$, $h\in\{1,\ldots,m_i\}$, $k\in \{1,...,T-1\}$ and for $\lambda = \nu$, $\gamma_x = \beta_x$ and $\gamma_u = \beta_u$. But in the equilibrium, Lemma 1 says that the reference trajectory $c^{*(i)}$ sent to each agent is exactly the one that would be obtained if each agent applied the input sequent $y^{*(i)}$. Hence $(x^{*},u^{*})$ solves, not only OCP-S, but also each agent's problem when $c^{*(i)}$ is sent to the agents (OCP-A). This can be seen directly by using the multipliers $\lambda^{(i)}$, $\gamma_x^{(i)}$ and $\gamma_u^{(i)}$ for every agent's subproblem and verifying that $(x^{*(i)},u^{*(i)})$ solves the KKT conditions of OCP-A. As a result, no agent has incentive to deviate from $m^{*}_{i}$. Hence $m^{*}$ will be a Nash equilibrium of the game induced by the mechanism.
\end{proof}

We finish this section with some remarks on Theorem 1:

At equilibrium, each agent reports the largest possible bounds $\tilde{J}^{*(i)}$ so that OCP-S is always feasible at equilibrium. One may argue why do we include such reports in the message vector? Their presence is key to establishing Lemma 2, as they provide a ``credible threat'' to the mechanism (and thus to other agents). This forces that the solution $(x(y^*),y^*)$ of OCP-S must solve each agent's subproblem at equilibrium. A similar argument with a numerical example is given in \cite{farhadi2018surrogate} in the context of routing. Also, each agent reports weights such that the derivative of the surrogate function $f$ matches exactly their marginal cost with respect to states and inputs. Secondly, observe that the fees payed in equilibrium are not zero, as they depend on the reference trajectory sent by the mechanism. The intuition behind this is that the fee charged to agent $i$ can capture the ``externality" cost it imposes to the system by having his cost function considered by the mechanism. Lastly, OCP-S may be infeasible outside of equilibrium, since an agent could report an infeasible operational range. This issue can be overcome by assuming that the principal may apply some feasible control input if OCP-S ends up being infeasible. More importantly, in order for the agents to behave according to the equilibrium strategies, they need to know the optimal solution $(x^{*},u^{*})$ for OCP-T. This means that the agents need to ``learn'' the equilibrium by replaying the game and refining their messages. In the next section, we will provide one such simple learning process and, instead of theoretically proving its convergence to the Nash equilibrium defined in Theorem 1, we will present a test case on HVAC control in an MPC setting, where the game is replayed consecutively, but at each time, the initial condition $\bar{x}_{0}$ is different. This showcases the potential use of our mechanism when a learning protocol is used within the MPC framework.


\section{HVAC Control Case Study}

\label{sec5}
\begin{figure*}\centering
 \includegraphics[trim=0.5in 0 0.5in 0,clip]{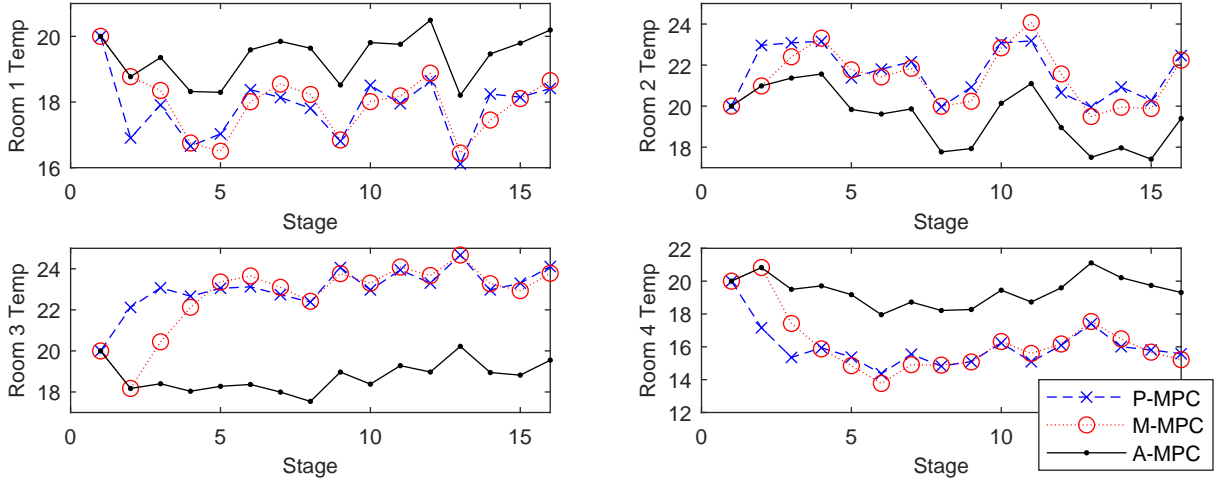}
\caption{Closed-Loop State Trajectories for (blue x-marked dashed line) P-MPC: Perfect Information Case; (red circle-marked dotted line) M-MPC: Surrogate-Mechanism Case; and (black dot-marked solid line) A-MPC: Consensus-Average Case}
\label{fig:state}
\end{figure*}

\begin{figure} \centering
 \scalebox{.35}{\includegraphics{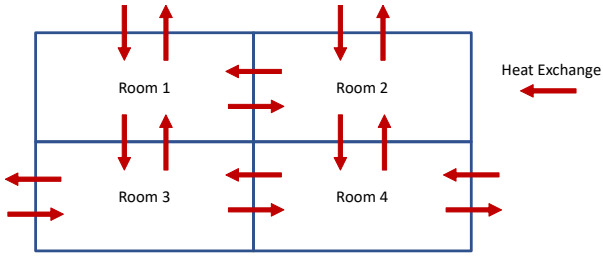}}
\caption{Room Configuration with Heat Exchange Vectors highlighted}
\label{fig:room}
\end{figure}

Consider a building manager who controls the HVAC system for four rooms. Each room occupant is an agent. Let $T_k = [T^{1}_k,T^{2}_k,T^{3}_k,T^{4}_k]^{\top}$ the state be the room temperatures. The building manager can heat/cool each individual room: Let $u_k = [u^1_k,u^2_k,u^3_k,u^4_k]^\top$ be the inputs in each room. Fig. \ref{fig:room} shows the layout of the rooms with respect to each other. Using standard HVAC models \cite{aswani2012identifying}, the dynamics are
\begin{equation}
\label{eqn:tdyn}
T_{k+1} = 
\begin{bmatrix}
\rho_1 & -\beta & - \gamma & 0 \\
-\beta & \rho_2 & 0 & \eta \\
-\gamma & 0 & \rho_3 & -\nu \\
0 & -\eta & -\nu & \rho_4
\end{bmatrix}T_k + \mu u_k- \alpha\begin{bmatrix}
T^{out}_{k} \\ T^{out}_{k} \\ T^{out}_{k} \\ T^{out}_{k}
\end{bmatrix} 
\end{equation} 
where $\rho_1 = 1 + \alpha + \beta + \gamma$; $\rho_2 =  1 + \alpha +\beta + \eta$; $\rho_3 = 1 + \alpha + \gamma + \nu$; $\rho_4 = 1 + \alpha + \eta + \nu$; and $\beta,\gamma,\eta,\nu$ are the heat transmission coefficients between rooms; and $\alpha$ is the heat coefficient with the outside. In addition, $\mu$ is the heat coefficient between the HVAC and each room. Note we treat the outside temperature as an exogenous disturbance vector. 

Now suppose each agent has the private cost function 
\begin{equation}
\textstyle\textbf{V}_{i}(x^{(i)},u^{(i)}) = \frac{\lambda_i}{2}\sum_{k=0}^{N-1} (T^{i}_{k} - T^{i}_{d})^2 + \frac{(1-\lambda_i)}{2}e^{(\gamma_iu^{i}_{k})^2}
\end{equation} 
where the tuple $(T^{i}_d,\lambda_i,\gamma_i)$ is the agent's private information, namely: their desired room temperate and two scalars regulating the trade-off between comfort and energy usage. Following the setup of our mechanism, the building manager does not know the agent's private information nor the shape of their objective functions. The manager broadcasts the function $f(T^i,u^i;v,w)  = \frac{1}{2}(e-vT_{r})^{2}+ \frac{1}{2}(wu)^{2}$, where $T_{r}$ is a reference temperature for the building manager. 

We consider an MPC setting, where the principal's receding horizon OCP-S at stage $t$ is given by
\begin{align} \label{P-MPC}
\min\ &\textstyle\frac{1}{2}\sum_{k =0}^{N-1}\sum_{i =1}^{I} (T^{i}_{t+k|t} - v^{(i)} _{t+k|t}T_{r})^2+ (w^{(i)} _{t+k|t}u^i_{t+k|t})^2\nonumber\\
\mathrm{s.t.}\ & T_{t+k+1|t} = AT_{t+k|t} + Bu_{t+k|t} + b_{t+k|t} \text{, } \forall k \in [T-1] \nonumber\\
&  (T^{(i)}_{t+k|t}u^{(i)}_{t+k|t}) \in \tilde{J}^{(i)}_{t}, \forall k \in [N-1]\\
& u_{min} \leq u^{(i)}_{t+k|t} \leq u_{max} ,\forall i \in [I], k \in [T-1]\nonumber\\
& T^{(i)}_{t|t} = T_t \text{ , } \forall i \in [I] \nonumber
\end{align}
where $A,B$ are given in (\ref{eqn:tdyn}) and $b_{t+k|t}$ is a prediction of $-\alpha T^{out}_{t+k}$ made at time $t$. Also, we use $u^{(i)}_{t+k|t}$ to denote the open-loop control input computed at stage $t$. Let $(T^{*}_{t},u^{*}_{t})$ be the optimal solution of (\ref{P-MPC}). The manager uses the current open-loop trajectories sent by agents to compute references
\begin{equation}
\textstyle c^{(i)}_{t+k|t} = \sum_{i \in\mathcal{N}_{i}} A_{ij}\tilde{T}^{(j)}_{t+k|t} \text{, } \forall k \in [T-1]
\end{equation} 
where the neighborhoods $\mathcal{N}_{i}$ match the room configurations. The principal also uses $T^{*(i)}_{t}$ in order to compute the reference trajectory in the fee $p_i$. After receiving such references, each agent solves their own OCP-A with the computed fees $p_i$ in the objective, obtaining a private solution vector $(\hat{T}^{(i)}_{t},\hat{u}^{(i)}_{t})$ and setting $\tilde{\lambda}^{(i)}_{t}$ to be the lagrange multipliers associated with the dynamics. Then each agent updates the remaining according to (\ref{N_mesg}), which in our case reduces to
\begin{equation}
\begin{aligned}
& v^{i}_{t+k|t} = ((1-\lambda_i) \hat{T}^{i}_{t+k|t} + T^{i}_{d})/T_{r} \\
& w^{i}_{t+k|t} = \sqrt{(1 - \lambda_i)\gamma^2_{i}e^{(\gamma_i\hat{u}_{t+k|t})^2}}
\end{aligned}
\end{equation}
Lastly, $\tilde{J}^{(i)}_{t} = (\{-\infty,+\infty\}_{k=0}^T, \{-\infty,+\infty\}_{k=0}^T)$ and $\tilde{T}^{i}_{t+1} = \hat{T}^{(i)}_{t}$. When $t=0$, all weights are initialized to unit values. All optimization problems were solved using the optimization solver MOSEK \cite{mosek2015mosek}. We consider a optimal control length $T=5$ and an MPC horizon of $N=15$. 
We compare our mechanism-based MPC (M-MPC) with the perfect-information case (P-MPC), where the principal knows the exact form of each $\textbf{V}_{i}$. We also consider a ``consensus''-type case, where no weights are updated and $T_{r}$ is set to the average of the desired temperatures (A-MPC). Fig. \ref{fig:state} shows the closed-loop state trajectory of the three approaches. It shows that our M-MPC closely tracks the P-MPC trajectory. Note that disturbance from the outside temperature causes the room temperatures to fluctuate around the desired values. 

Fig. \ref{fig:cost} shows M-MPC recovers the P-MPC cost after a few time steps. Since we used true costs to compute P-MPC, this shows our mechanism recovers the efficient trajectory. In contrast, the case without information exchange behaves poorly. This example shows our mechanism can be used with MPC: at each stage an optimal control problem is solved, the first-stage control is applied, and agents update their messages based on knowledge received from the principal. 
 
 \begin{figure} \centering
  \includegraphics{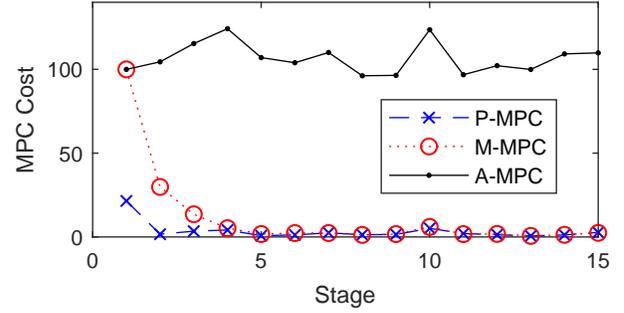}
 \caption{MPC Aggregated Stage Cost with Agents' True Utility Functions}
 \label{fig:cost}
 \end{figure}
\section{Conclusion}

We studied a dynamical system with several non-cooperative strategic agents. We proposed a mechanism where the agents interact via a platform and characterized the equilibrium strategies. We provided an HVAC control test case to highlight the need of designing mechanisms that have low-communication requirements in an MPC setting.

\bibliographystyle{IEEEtran}
\bibliography{IEEEabrv,MechDes}

\end{document}